\documentclass{amsart}%
\usepackage{amsfonts}
\usepackage{amsmath}
\usepackage{amssymb}
\usepackage{graphicx}%
\providecommand{\U}[1]{\protect\rule{.1in}{.1in}}
\newtheorem{theorem}{Theorem}
\theoremstyle{plain}

\newtheorem{corollary}{Corollary}

\newtheorem{example}{Example}

\newtheorem{lemma}{Lemma}

\newtheorem{remark}{Remark}

\numberwithin{equation}{section}
\begin{document}
\title[ ]{Quantitative estimates of convergence in nonlinear operator extensions of
Korovkin's theorems}
\author{Sorin G. Gal}
\address{Department of Mathematics and Computer Science,\\
University of Oradea, Romania and Academy of Romanian Scientists, Bucharest, Romania}
\email{galso@uoradea.ro, galsorin23@gmail.com}
\author{Constantin P. Niculescu}
\address{Department of Mathematics, University of Craiova, Romania}
\email{constantin.p.niculescu@gmail.com}
\date{October 19, 2022}

\begin{abstract}
This paper is aimed to prove a quantitative estimate (in terms of the modulus
of continuity) for the convergence in the nonlinear version of Korovkin's
theorem for sequences of weakly nonlinear and monotone operators defined on
spaces of continuous real functions. Several examples illustrating the theory
are included.

\end{abstract}
\maketitle

\section{Introduction}

In its original form, Korovkin's theorem \cite{Ko1953}, \cite{Ko1960} provides
a very simple test of convergence to the identity for any sequence
$(L_{n})_{n}$ of positive linear operators that map $C\left(  [0,1]\right)  $
into itself: the occurrence of this convergence for the functions $1,~x$ and
$x^{2}$. In other words, the fact that%
\[
\lim_{n\rightarrow\infty}L_{n}(f)=f\text{\quad uniformly on }[0,1]
\]
for every $f\in C\left(  [0,1]\right)  $ reduces to the status of the three
aforementioned functions. Due to its simplicity and usefulness, this result
has attracted a great deal of attention leading to numerous generalizations.
Part of them are included in the authoritative monograph of Altomare and
Campiti \cite{AC1994} and the excellent survey of Altomare \cite{Alt2010}.

In a series of papers published since 2020, we have extended Korovkin's
theorem to framework of operators weakly nonlinear and monotone operators
acting on function spaces. See \cite{Gal-Nic-Med}, \cite{Gal-Nic-RACSAM},
\cite{Gal-Nic-Med2023} and \cite{Gal-Nic-subm}. For reader's convenience, we
recall here that an operator $T:E\rightarrow F$ is called \emph{weakly
nonlinear} if it satisfies the following two conditions:

\begin{enumerate}
\item[(SL)] (\emph{Sublinearity}) $T$ is subadditive and positively
homogeneous, that is,%
\[
T(f+g)\leq T(f)+T(g)\quad\text{and}\quad T(\alpha f)=\alpha T(f)
\]
for all $f,g$ in $E$ and $\alpha\geq0;$

\item[(TR)] (\emph{Translatability}) $T(f+\alpha\cdot1)=T(f)+\alpha T(1)$ for
all functions $f\in E$ and all numbers $\alpha\geq0.$
\end{enumerate}

In the case where $T$ is \emph{unital} (that is, $T(1)=1)$ the condition of
translatability takes the form%
\[
T(f+\alpha\cdot1)=T(f)+\alpha1,
\]
for all $f\in E$ and $\alpha\geq0.$

A stronger condition than translatability is

\begin{enumerate}
\item[(TR$^{\ast}$)] (\emph{Strong translatability}) $T(f+\alpha
\cdot1)=T(f)+\alpha T(1)$ for all functions $f\in E$ and all numbers
$\alpha\in\mathbb{R}.$
\end{enumerate}

The last condition occurs naturally in the context of Choquet's integral,
being a consequence of what is called there the property of \emph{comonotonic
additivity}, that is,

\begin{enumerate}
\item[(CA)] $T(f+g)=T(f)+T(g)$ whenever the functions $f,g\in E$ are
comonotone in the sense that%
\[
(f(s)-f(t))\cdot(g(s)-g(t))\geq0\text{\quad for all }s,t\in X.
\]
See \cite{Gal-Nic-Aeq} and \cite{Gal-Nic-JMAA}, as well as the references therein.
\end{enumerate}

In this paper we are especially interested in those weakly nonlinear operators
which preserve the ordering, that is, which verify the following condition:

\begin{enumerate}
\item[(M)] (\emph{Monotonicity}) $f\leq g$ in $E$ implies $T(f)\leq T(g)$ for
all $f,g$ in $E.$
\end{enumerate}

Examples of weakly nonlinear and monotone operators based on Choquet's theory
of integration can be found in \cite{Gal-Nic-Aeq}, \cite{Gal-Nic-JMAA} and
\cite{Gal-Nic-RACSAM}. Ergodic theory and harmonic analysis offer numerous
other examples of nonlinear operators which are sublinear, monotone and
strongly translatable. Here is an example.

Suppose that $E$ is a vector lattice of functions that contains the unity and
$U:E\rightarrow E$ is a sublinear, (strongly) translatable, monotone and
unital operator. Then each of the Yosida-Kakutani operators
\[
YK_{n}:E\rightarrow E,\quad YK_{n}(f)=\sup\left\{  f,\frac{1}{2}\left(
f+Uf\right)  ,...,\frac{1}{n}\sum_{k=0}^{n-1}U^{k}f\right\}  \text{\quad}%
(n\in\mathbb{N})
\]
verifies the same string of properties. In the case where $E=L^{p}(\mu)$ (for
$\mu$ a probability measure and $1\leq p<\infty$) and $U$ verifies in addition
the condition%
\[
\int\left\vert U(f)\right\vert ^{p}d\mu\leq C\int\left\vert f\right\vert
^{p}d\mu\text{\quad for all }f,
\]
then the operators $YK_{n}$ appear as truncations of the operator%
\[
YK:E\rightarrow E,\quad YK(f)=\sup\left\{  f,\frac{1}{2}\left(  f+Uf\right)
,...,\frac{1}{n}\sum_{k=0}^{n-1}U^{k}f,...\right\}  ,
\]
which is also sublinear, (strongly) translatable, monotone and unital. The
operators $YK_{n}$ and $YK$ came from the maximal ergodic theorem of Yosida
and Kakutani \cite{YK}.

Following an idea due to Popa \cite{Popa2022}, we have proved in
\cite{Gal-Nic-subm} the following Korovkin type theorem for sequences of
weakly nonlinear and monotone operators converging to an operator possibly
different from the identity:

\begin{theorem}
\label{thm1}Suppose that $K$ is a compact subset of the Euclidean orthant
$\mathbb{R}_{+}^{N}$, $X$ is a compact Hausdorff space and $T_{n}$
$(n\in\mathbb{N})$ and $A$ are weakly nonlinear and monotone operators from
$E=C(K)$ into $C(X)$ such that
\begin{equation}
A(1)\text{ is a strictly positive element} \label{hyp>0}%
\end{equation}
and
\begin{equation}
A(1)A(\sum_{k=1}^{N}\left(  \operatorname*{pr}\nolimits_{k}\right)  ^{2}%
)=\sum_{k=1}^{N}(A(-\operatorname*{pr}\nolimits_{k}))^{2}. \label{hypA}%
\end{equation}

Then%
\[
\lim_{n\rightarrow\infty}\left\Vert T_{n}(f)-A(f)\right\Vert =0\text{\quad for
all }f\in C(K)
\]
if and only if this property of convergence occurs for each of the functions
\begin{equation}
1,-\operatorname*{pr}\nolimits_{1},...,-\operatorname*{pr}\nolimits_{N}\text{
and }\sum\nolimits_{k=1}^{N}\operatorname*{pr}\nolimits_{k}^{2}.
\label{testset}%
\end{equation}

\end{theorem}

Here $\operatorname*{pr}\nolimits_{k}$ denotes the canonical projection on the
$k$th coordinate.

A more general result that works for all compact subsets of the Euclidean
space $\mathbb{R}^{N}$ is available in \cite{Gal-Nic-subm}, Theorem 3.

In a very influential paper, Shisha and Mond \cite{SM} have put the classical
theorem of Korovkin in a quantitative form, expressing the rate of convergence
of $L_{n}(f)$ to $f$, in terms of the rates of convergence of $L_{n}(l)$ to
$1$, $L_{n}(x)$ to $x$, and $L_{n}(x^{2})$ to $x^{2}$.\ The aim of the present
paper is to prove a similar result covering the nonlinear framework of Theorem
1. See Theorem \ref{thm2}, Section 3,\ which represents the quantitative form
of Theorem \ref{thm1}. As was mentioned in \cite{Gal-Nic-subm}, Theorem
\ref{thm1} admits a number of trigonometric variants working for the
continuous functions defined on the unit circle $S^{1}$ (or on the
$2$-dimensional sphere $S^{2},$ on the $2$-dimensional torus $S^{1}\times
S^{1}$ etc.) Their quantitative forms need slight modifications of the
argument of Theorem 2 below and we leave the details as an exercise.

Applications are presented in Section 4. For the convenience of the reader, we
summarized in Section 2 some facts concerning the norm of a continuous and
sublinear operator.

\section{Background on weakly nonlinear operators}

Suppose that $E$ and $F$ are two ordered Banach spaces and $T$ $:E\rightarrow
F$ is an operator (not necessarily linear or continuous).

If $T$ is positively homogeneous, then
\[
T(0)=0.
\]
As a consequence,
\[
-T(-f)\leq T(f)\text{\quad for all }f\in E
\]
and every positively homogeneous and monotone operator $T$ maps positive
elements into positive elements, that is,%
\begin{equation}
Tf\geq0\text{\quad for all }f\geq0. \label{pos-op}%
\end{equation}
Therefore, for linear operators, the property (\ref{pos-op}) is equivalent to monotonicity.

The \emph{norm} of a continuous sublinear operator $T:E\rightarrow F$ can be
defined via the formulas%
\begin{align*}
\left\Vert T\right\Vert  &  =\inf\left\{  \lambda>0:\left\Vert T\left(
f\right)  \right\Vert \leq\lambda\left\Vert f\right\Vert \text{ for all }f\in
E\right\} \\
&  =\sup\left\{  \left\Vert T(f)\right\Vert :f\in E,\text{ }\left\Vert
f\right\Vert \leq1\right\}  .
\end{align*}
A sublinear operator may be discontinuous, but when it is continuous, it is
Lipschitz continuous. More precisely, if $T:E\rightarrow F$ is a continuous
sublinear operator, then
\[
\left\Vert T\left(  f\right)  -T(g)\right\Vert \leq2\left\Vert T\right\Vert
\left\Vert f-g\right\Vert \text{\quad for all }f\in E.
\]

Remarkably, all sublinear and monotone operators are Lipschitz continuous:

\begin{lemma}
\label{thmKrein}Every sublinear and monotone operator $T$ $:E\rightarrow F$
verifies the inequality
\[
\left\vert T(f)-T(g)\right\vert \leq T\left(  \left\vert f-g\right\vert
\right)  \text{\quad for all }f,g\in E
\]
and thus it is Lipschitz continuous with Lipschitz constant equals to
$\left\Vert T\right\Vert ,$ that is,
\[
\left\Vert T(f)-T(g)\right\Vert \leq\left\Vert T\right\Vert \left\Vert
f-g\right\Vert \text{\quad for all }f,g\in E.
\]

\end{lemma}

See \cite{Gal-Nic-Med2023} for details. Lemma \ref{thmKrein} is a
generalization of a classical result of M. G. Krein concerning the continuity
of positive linear functionals. See \cite{AA2001} for a historical account.

\section{The main results}

Before to state our main result we need to recall a basic inequality from
Shisha and Mond \cite{SM}, p. 1197, concerning the modulus of continuity,
\[
\omega(f,\delta)=\sup\{|f(x)-f(y)|:x,y\in K\text{, }\Vert x-y\Vert\leq
\delta\},\quad\delta>0,
\]
of a real-valued continuous functions $f$ defined on compact subset $K$ of
$\mathbb{R}^{N}.$ For convenience, we adopt the convention that $\omega
(f,0)=0.$

\begin{lemma}
\label{lem2}For all $x,y\in K$ and $\delta>0,$
\[
|f(x)-f(y)|\leq(1+\Vert x-y\Vert^{2}\delta^{-2})\omega(f,\delta).
\]

\begin{proof}
Indeed, if $\Vert x-y\Vert\leq\delta$, then
\[
|f(x)-f(y)|\leq\omega(f,\Vert x-y\Vert)\leq\omega(f,\delta)\leq(1+\Vert
x-y\Vert^{2}\delta^{-2})\omega(f,\delta),
\]
while if $\Vert x-y>\delta$, we have to take into account the following
well-known property of the modulus of continuity,
\[
\omega(f,\lambda\delta)\leq(1+\lambda)\omega(f,\delta)\quad\text{ for }%
\lambda,\delta>0,
\]
which yields
\begin{align*}
|f(x)-f(y)|  &  \leq\omega(f,\Vert x-y\Vert)=\omega\left(  f,\delta\cdot
\frac{\Vert x-y\Vert}{\delta}\right) \\
&  \leq(1+\delta^{-1}\Vert x-y\Vert)\omega(f,\delta)\leq(1+\Vert x-y\Vert
^{2}\delta^{-2})\omega(f,\delta).
\end{align*}
The proof is done.
\end{proof}
\end{lemma}

We can now state the main result of our paper:

\begin{theorem}
\label{thm2}Suppose that $K$ is a compact subset of the Euclidean orthant
$\mathbb{R}_{+}^{N}$, $X$ is a compact Hausdorff space and $T_{n}$
$(n\in\mathbb{N})$ and $A$ are weakly nonlinear and monotone operators from
$E=C(K)$ into $C(X)$ such that $A(1)$ is a strictly positive element. Put
\[
M=1/\inf_{x\in X}(A(1)(x))
\]
and
\[
\mu_{n}=\left\Vert T_{n}\left(  \sum_{k=1}^{N}\operatorname*{pr}%
\nolimits_{k}^{2}\right)  A(1)-2\sum_{k=1}^{N}A\left(  -\operatorname*{pr}%
\nolimits_{k}\right)  T_{n}\left(  -\operatorname*{pr}\nolimits_{k}\right)
+A\left(  \sum_{k=1}^{N}\operatorname*{pr}\nolimits_{k}^{2}\right)
T_{n}(1)\right\Vert ^{1/2}.
\]
Then
\begin{equation}
\Vert T_{n}(f)-A(f)\Vert\leq M\left\{  {\Vert T_{n}(1)-A(1)\Vert\cdot\Vert
A(f)\Vert+(\Vert T_{n}(1)A(1)\Vert+1)\omega(f,\mu_{n})}\right\}  , \label{eq0}%
\end{equation}
for all $f\in C(K)$ and all $n\in\mathbb{N}$.

In the particular case where $T_{n}(1)=A(1)=1$, the above estimate reduces to
\[
\Vert T_{n}(f)-A(f)\Vert\leq2\omega(f,\mu_{n}).
\]

\end{theorem}

\begin{proof}
Since $A(1)(X)$ is a compact subset of $\mathbb{R}$ and $A(1)(x)>0$ for all
$x\in X$, it follows that $\inf_{x\in X}A(1)(x)>0.$ Then for all $f\in C(X),$
\begin{align}
\frac{1}{M}\Vert T_{n}(f)-A(f)\Vert &  =\inf_{x\in X}A(1)(x)\cdot\sup_{x\in
X}\left\vert T_{n}(f)(x)-A(f)(x)\right\vert \label{basic_estimate}\\
&  \leq\sup_{x\in X}\left\vert \left(  T_{n}(f)(x)-A(f)(x)\right)
A(1)(x)\right\vert \nonumber\\
&  =\Vert(T_{n}(f)-A(f))A(1)\Vert,\nonumber
\end{align}
which reduces the proof of the inequality to that of the inequality%
\begin{equation}
\Vert(T_{n}(f)-A(f))A(1)\Vert\leq\left\{  {\Vert T_{n}(1)-A(1)\Vert\cdot\Vert
A(f)\Vert+(\Vert T_{n}(1)A(1)\Vert+1)\omega(f,\mu_{n})}\right\}  .
\label{fin_eq}%
\end{equation}

For this one, according to Lemma \ref{lem2},
\[
|f(x)-f(y)|\leq(1+\Vert x-y\Vert^{2}\delta^{-2})\omega(f,\delta),
\]
for all $x,y\in K$ and $\delta>0$ a fact that yields, viewing $y$ as a
parameter, the inequality%
\[
\left\vert f-f(y)\right\vert \leq\left[  1+\delta^{-2}\left(  \sum_{k=1}%
^{N}\operatorname*{pr}\nolimits_{k}^{2}+2\sum_{k=1}^{N}\operatorname*{pr}%
\nolimits_{k}(y)\operatorname*{pr}\nolimits_{k}\right.  \left.  +\sum
_{k=1}^{N}\operatorname*{pr}\nolimits_{k}^{2}(y)\cdot1\right)  \right]
\omega(f,\delta).
\]
Suppose for a moment that $f\geq0$. Then, by taking into account Lemma
\ref{thmKrein},
\begin{multline*}
|T_{n}(f)-f(y)T_{n}(1)|\leq T_{n}(|f-f(y)\cdot1|)\\
\leq\left[  T_{n}(1)+\delta^{-2}\left(  T_{n}(\sum_{k=1}^{N}\left(
\operatorname*{pr}\nolimits_{k}\right)  ^{2})+2\sum_{k=1}^{N}\left(
\operatorname*{pr}\nolimits_{k}(y)\right)  T_{n}\left(  -\operatorname*{pr}%
\nolimits_{k}\right)  \right.  \right. \\
\left.  \left.  +\sum_{k=1}^{N}\left(  \operatorname*{pr}\nolimits_{k}%
(y)\right)  ^{2}T_{n}(1)\right)  \right]  \omega(f,\delta),
\end{multline*}
which leads, for each $x\in X,$ to the following inequality in $C(X):$
\begin{multline*}
|\left(  T_{n}(f)\right)  (x)-T_{n}(1)(x)\cdot f|\\
\leq\left[  T_{n}(1)(x)+\delta^{-2}\left(  T_{n}(\sum_{k=1}^{N}%
\operatorname*{pr}\nolimits_{k}^{2})(x)+2\sum_{k=1}^{N}T_{n}\left(
-\operatorname*{pr}\nolimits_{k}\right)  (x)\right.  \right.
\operatorname*{pr}\nolimits_{k}\\
\left.  \left.  +\sum_{k=1}^{N}T_{n}(1)(x)\operatorname*{pr}\nolimits_{k}%
^{2}\right)  \right]  \omega(f,\delta).
\end{multline*}
Applying the weakly nonlinear and monotone operator $A$ to the both sides of
this last inequality and noticing that $-2T_{n}(-\operatorname*{pr}%
\nolimits_{k})\geq0,$ we obtain
\begin{multline*}
|\left(  T_{n}(f)\right)  (x)A(1)-\left(  T_{n}(1)\right)  (x)A(f)|\leq
A\left(  |\left(  T_{n}(f)\right)  (x)-\left(  T_{n}(1)\right)  (x)\cdot
f|\right) \\
\leq\left[  T_{n}(1)(x)A(1)+\delta^{-2}\left(  T_{n}(\sum_{k=1}^{N}%
\operatorname*{pr}\nolimits_{k}^{2})(x)A(1)-2\sum_{k=1}^{N}A\left(
-\operatorname*{pr}\nolimits_{k}\right)  T_{n}\left(  -\operatorname*{pr}%
\nolimits_{k}\right)  (x)\right.  \right. \\
\left.  \left.  +A(\sum_{k=1}^{N}\operatorname*{pr}\nolimits_{k}^{2}%
)T_{n}(1)(x)\right)  \right]  \omega(f,\delta),
\end{multline*}
whence
\begin{multline*}
|\left(  T_{n}(f)\right)  (x)A(1)(x)-\left(  T_{n}(1)\right)  (x)A(f)(x)|\\
\leq\left[  T_{n}(1)(x)A(1)(x)+\delta^{-2}\left(  T_{n}(\sum_{k=1}%
^{N}\operatorname*{pr}\nolimits_{k}^{2})(x)A(1)(x)\right.  \right. \\
-2\sum_{k=1}^{N}A\left(  -\operatorname*{pr}\nolimits_{k}\right)
(x)T_{n}(-\operatorname*{pr}\nolimits_{k})(x)\left.  \left.  +A(\sum_{k=1}%
^{N}\operatorname*{pr}\nolimits_{k}^{2})(x)T_{n}(1)(x)\right)  \right]
\omega(f,\delta),
\end{multline*}
for all $x\in X.$ Now, taking the supremum over $x\in K$, we arrive at
\begin{equation}
\Vert\left(  T_{n}(f)\right)  A(1)-\left(  T_{n}(1)\right)  A(f)\Vert
\label{eq3.1}%
\end{equation}%
\begin{multline*}
\leq\left[  \Vert T_{n}(1)A(1)\Vert+\delta^{-2}\left(  \left\Vert T_{n}%
(\sum_{k=1}^{N}\operatorname*{pr}\nolimits_{k}^{2})A(1)-2\sum_{k=1}%
^{N}A\left(  -\operatorname*{pr}\nolimits_{k}\right)  T_{n}\left(
-\operatorname*{pr}\nolimits_{k}\right)  \right.  \right.  \right. \\
\left.  \left.  +\left.  A(\sum_{k=1}^{N}\operatorname*{pr}\nolimits_{k}%
^{2})T_{n}(1)\right\Vert \right)  \right]  \omega(f,\delta).
\end{multline*}

The case where $f\in C(X)$ is a signed function can be reduced to the
precedent one by replacing $f$ by $f+\left\Vert f\right\Vert _{\infty}.$
Indeed, using the property of weak nonlinearity of the operators $L_{n}$ and
$A$ we have%
\begin{align*}
&  T_{n}(f+\left\Vert f\right\Vert _{\infty})A(1)-A(f+\left\Vert f\right\Vert
_{\infty})T_{n}(1)\\
&  =\left(  T_{n}(f)+\left\Vert f\right\Vert _{\infty}T_{n}(1)\right)
A(1)-\left(  A(f)+\left\Vert f\right\Vert _{\infty}A(1)\right)  T_{n}(1)\\
&  =T_{n}(f)A(1)-A(f)T_{n}(1),
\end{align*}
which makes the inequality (\ref{eq3.1}) valid for all functions $f\in C(X).$

Assuming $\mu_{n}\not =0,$ we infer from the inequality (\ref{eq3.1}) applied
for
\[
\delta=\mu_{n}=\left\Vert T_{n}(\sum_{k=1}^{N}\operatorname*{pr}%
\nolimits_{k}^{2})A(1)-2\sum_{k=1}^{N}A\left(  -\operatorname*{pr}%
\nolimits_{k}\right)  T_{n}\left(  -\operatorname*{pr}\nolimits_{k}\right)
+A(\sum_{k=1}^{N}\operatorname*{pr}\nolimits_{k}^{2})T_{n}(1)\right\Vert
^{1/2},
\]
the relation
\[
\Vert\left(  T_{n}(f)\right)  A(1)-\left(  T_{n}(1)\right)  A(f)\Vert
\leq(\Vert T_{n}(1)A(1)\Vert+1)\omega(f,\mu_{n}).
\]
Then
\begin{multline*}
\Vert T_{n}(f)A(1)-A(1)A(f)\Vert\leq\Vert T_{n}(f)A(1)-T_{n}(1)A(f)\Vert+\Vert
T_{n}(1)A(f)-A(1)A(f)\Vert\\
\leq(\Vert T_{n}(1)A(1)\Vert+1)\omega(f,\mu_{n})+\Vert A(f)\Vert\cdot\Vert
T_{n}(1)-A(1)\Vert
\end{multline*}
and the proof of the inequality (\ref{fin_eq}) is done.

It remains to consider the case $\mu_{n}=0$. In this case the inequality
(\ref{eq3.1}) reduces to
\[
\Vert T_{n}(f)A(1)-T_{n}(1)A(f)\Vert=0.
\]
Then
\begin{align*}
\Vert T_{n}(f)A(1)-A(1)A(f)\Vert &  \leq\Vert T_{n}(f)A(1)-T_{n}%
(1)A(f)\Vert+\Vert T_{n}(1)A(f)-A(1)A(f)\Vert\\
&  \leq\Vert A(f)\Vert\cdot\Vert T_{n}(1)-A(1)\Vert,
\end{align*}
which coincides with the assertion of inequality (\ref{eq0}) for $\mu_{n}=0.$

The proof is done.
\end{proof}

\begin{remark}
\label{rem1}Theorem \emph{\ref{thm2} }implies\emph{ }Theorem \emph{\ref{thm1}%
.} To prove this, we have to notice that
\begin{multline*}
T_{n}(\sum_{k=1}^{N}\operatorname*{pr}\nolimits_{k}^{2})A(1)-2\sum_{k=1}%
^{N}A\left(  -\operatorname*{pr}\nolimits_{k}\right)  T_{n}\left(
-\operatorname*{pr}\nolimits_{k}\right)  +A(\sum_{k=1}^{N}\operatorname*{pr}%
\nolimits_{k}^{2})T_{n}(1)\\
=[T_{n}(\sum_{k=1}^{N}\operatorname*{pr}\nolimits_{k}^{2})-A(\sum_{k=1}%
^{N}\operatorname*{pr}\nolimits_{k}^{2})]A(1)\\
-2\sum_{k=1}^{N}A(-\operatorname*{pr}\nolimits_{k})[T_{n}(-\operatorname*{pr}%
\nolimits_{k})-A(-\operatorname*{pr}\nolimits_{k})]\\
+[T_{n}(1)-A(1)]A(\sum_{k=1}^{N}\operatorname*{pr}\nolimits_{k}^{2})+2\Delta,
\end{multline*}
where $\Delta=A(1)A(\sum_{k=1}^{N}\operatorname*{pr}\nolimits_{k}^{2}%
)-(\sum_{k=1}^{N}A(-\operatorname*{pr}\nolimits_{k}))^{2}$. Therefore, if
$\Delta=0$ $($which is one of the assumptions of Theorem $1)$, it follows that
$\mu_{n}\rightarrow0$ whenever
\begin{align*}
\lim_{n\rightarrow\infty}T_{n}(-pr_{k})  &  =A(-pr_{k})\text{\quad for
}k=1,...,N\\
\lim_{n\rightarrow\infty}T_{n}(\sum_{k=1}^{N}pr_{k})  &  =A(\sum_{k=1}%
^{N}(pr_{k})^{2})
\end{align*}
and $\lim_{n\rightarrow\infty}T_{n}(1)=A(1)$, uniformly on $X$. According to
Theorem \ref{thm2},
\[
\lim_{n\rightarrow\infty}T_{n}(f)=A(f)
\]
uniformly on $X$ and the proof is done.
\end{remark}

In the case of functions $f\in C[0,1]$, Theorem \ref{thm2} can be stated as follows.

\begin{corollary}
\label{cor1} Denote $e_{k}(x)=x^{k}$ for $k\in\{0,1,2\}$ and let
$A:C([0,1])\rightarrow C([0,1])$ be a weakly nonlinear and monotone operator
such that $A(1)(x)>0$ for all $x\in\lbrack0,1]$. If $T_{n}:C([0,1])\rightarrow
C([0,1])$ is a sequence of weakly nonlinear and monotone operators, then for
every $f\in C([0,1])$ and $n=1,2,...,$ we have
\[
\Vert T_{n}(f)-A(f)\Vert\leq M\{{\Vert T_{n}(1)-A(1)\Vert\cdot\Vert
A(f)\Vert+(\Vert T_{n}(1)A(1)\Vert+1)\omega(f,\mu_{n})}\},
\]
where $M=[(\inf_{x\in\lbrack0,1]}A(1)(x))^{-1}]$ and
\[
\mu_{n}=\Vert T_{n}(e_{2})A(1)-2T_{n}(-e_{1})A(-e_{1})+T_{n}(1)A(e_{2}%
)\Vert^{1/2}.
\]
If, in addition, $T_{n}(1)=A(1)$, then the conclusion becomes
\[
\Vert T_{n}(f)-A(f)\Vert\leq M(\Vert A(1)^{2}\Vert+1)\omega(f,\mu_{n}).
\]

\end{corollary}

Corollary \ref{cor1} extends Theorem 1 from Shisha and Mond \cite{SM} (which
represents the particular case where all operators $T_{n}$ are linear,
continuous and positive and $A$ is the identity of $C([0,1])).$

\section{Applications}

In this section we apply Corollary \ref{cor1} to some concrete examples.

We will need the following family of polynomials
\[
p_{n,k}(x)={\binom{n}{k}}x^{k}(1-x)^{n-k},\text{\quad}0\leq k\leq n,
\]
related to Bernstein's proof of the Weierstrass approximation theorem. As is
well known they verify a number of combinatorial identities such as%
\begin{align}
\sum_{k=0}^{n}\binom{n}{k}x^{k}(1-x)^{n-k}  &  =1\label{comb1}\\
\sum_{k=0}^{n}k\binom{n}{k}x^{k}(1-x)^{n-k}  &  =nx\label{comb2}\\
\sum_{k=0}^{n}k^{2}\binom{n}{k}x^{k}(1-x)^{n-k}  &  =nx(1-x+nx) \label{comb3}%
\end{align}
for all $x\in\mathbb{R}.$ See, e.g., \cite{CN2014}, Theorem 8.8.1, p. 256.

\begin{example}
\label{ex1} Let $\varphi:[0,1]\rightarrow\lbrack0,1]$ be a continuous
function. Attached to it is the sequence of Bernstein-type operators
$B_{n}:C([0,1])\rightarrow C\left(  [0,1]\right)  ,$ defined by the formulas%
\[
B_{n}(f)(x)=\sum_{k=0}^{n}p_{n,k}(\varphi(x))f(k/n)
\]
and also the operator
\[
A:C\left(  [0,1]\right)  \rightarrow C\left(  [0,1]\right)  ,\text{\quad
}A(f)=f\circ\varphi.
\]
Clearly, these operators are linear, positive and unital. Also, the operator
$A$ verifies the condition
\[
A(1)A(x^{2})=\sum_{k=1}^{N}(A(x))^{2}.
\]
The operators $T_{n}:C([0,1])\rightarrow C\left(  [0,1]\right)  $ given by
\[
T_{n}(f)=\max\left\{  B_{n}(f),B_{n+1}(f)\right\}
\]
are sublinear, monotone and strongly translatable. Simple computations
\emph{(}based on the identities \emph{(\ref{comb1})-(\ref{comb3})), }show
that
\begin{align*}
B_{n}(1)  &  =1,\\
B_{n}(-x)  &  =-\varphi(x)
\end{align*}
and%
\[
B_{n}(x^{2})=(\varphi(x))^{2}+\frac{\varphi(x)(1-\varphi(x))}{n}.
\]
We get $T_{n}(1)=1$, $T_{n}(-x)=-\varphi(x)$ and $T_{n}(x^{2})=\max
\{B_{n}(x^{2}),B_{n+1}(x^{2})\}=(\varphi(x))^{2}+\frac{\varphi(x)(1-\varphi
(x))}{n}$, which by Corollary \emph{\ref{cor1}} implies
\[
\mu_{n}^{2}=\left\Vert \frac{\varphi(1-\varphi)}{n}\right\Vert \leq\frac
{1}{4n}%
\]
and therefore
\[
\Vert T_{n}(f)-A(f)\Vert\leq2\omega\left(  f,\frac{1}{2\sqrt{n}}\right)  .
\]

\end{example}

\begin{example}
\label{ex2} Given a continuous function $\varphi:[0,1]\rightarrow\lbrack0,1]$,
one associates to it the sequence of nonlinear operators $T_{n}:C\left(
[0,1]\right)  \rightarrow C\left(  [0,1]\right)  $ defined by the formulas
\[
T_{n}(f)(x)=\sum_{k=0}^{n}p_{n,k}(\varphi(x))\sup_{[k/(n+1)\leq t\leq
,(k+1)/(n+1)]}f(t).
\]
Also, define the linear, positive and unital operator
\[
A:C\left(  [0,1]\right)  \rightarrow C\left(  [0,1]\right)  ,\text{\quad
}A(f)=f\circ\varphi.
\]
Since
\begin{align*}
B_{n}(1)  &  =1,\\
B_{n}(-x)  &  =-\varphi(x)
\end{align*}
and%
\[
B_{n}(x^{2})=\varphi(x)^{2}+\frac{\varphi(x)(1-\varphi(x))}{n},
\]
we get $T_{n}(1)=1$,
\[
T_{n}(-x)=\sum_{k=0}^{n}p_{n,k}(\varphi(x))\cdot(-k/(n+1))=-\frac{n}{n+1}%
\sum_{k=0}^{n}p_{n,k}(\varphi(x))\frac{k}{n}=-\frac{n}{n+1}\varphi(x),
\]%
\[
T_{n}(x^{2})=\sum_{k=0}^{n}p_{n,k}(\varphi(x))\cdot\frac{(k+1)^{2}}{(n+1)^{2}%
}=\sum_{k=0}^{n}p_{n,k}(\varphi(x))\cdot\frac{k^{2}+2k+1}{(n+1)^{2}}%
\]%
\[
=\left(  \frac{n}{n+1}\right)  ^{2}\cdot\sum_{k=0}^{n}p_{n,k}(\varphi
(x))\frac{k^{2}}{n^{2}}+2\frac{n}{(n+1)^{2}}\sum_{k=0}^{n}p_{n,k}%
(\varphi(x))\frac{k}{n}+\frac{1}{(n+1)^{2}}%
\]%
\[
=\left(  \frac{n}{n+1}\right)  ^{2}\left(  (\varphi(x))^{2}+\frac
{\varphi(x)(1-\varphi(x))}{n}\right)  +\frac{2n}{(n+1)^{2}}\varphi(x)+\frac
{1}{(n+1)^{2}},
\]
which by Corollary \emph{\ref{cor1}} implies
\[
\mu_{n}^{2}=\left\Vert \left(  \frac{n}{n+1}\right)  ^{2}\left(  \varphi
^{2}+\frac{\varphi(1-\varphi)}{n}\right)  +\frac{2n\varphi}{(n+1)^{2}}%
+\frac{1}{(n+1)^{2}}-\frac{2n\varphi^{2}}{n+1}+\varphi^{2}\right\Vert
\]%
\[
=\left\Vert \frac{3n\varphi+1-\varphi^{2}(n-1)}{(n+1)^{2}}\right\Vert
\leq\left\Vert \frac{3n+1+(n-1)}{(n+1)^{2}}\right\Vert \leq\frac{4}{n}.
\]
and therefore
\[
\Vert T_{n}(f)-A(f)\Vert\leq2\omega\left(  f,\frac{2}{\sqrt{n}}\right)  .
\]

\end{example}

\textbf{Funding} The authors declare that no funds, grants, or other support
were received during the preparation of this manuscript.

\textbf{Competing Interests} The authors have no relevant financial or
non-financial interests to disclose.

\end{document}